\documentclass{article}
\usepackage{graphicx} 
\usepackage[english]{babel}
\usepackage{amssymb}
\usepackage[backend=biber,style=numeric,sorting=ynt]{biblatex}
\usepackage[letterpaper,top=2cm,bottom=2cm,left=3cm,right=3cm,marginparwidth=1.75cm]{geometry}

\usepackage{amsthm}
\usepackage
{thmbox}
\usepackage{amsmath}
\usepackage{graphicx} 
\usepackage[colorlinks=true, allcolors=blue]{hyperref}
\usepackage{stmaryrd}
\usepackage{algpseudocode}
\usepackage{algorithm}
\usepackage{tikz-cd}
\usepackage{float}
\usepackage{xfrac}
\usepackage{blindtext}
\usepackage{comment}
\usepackage{appendix}
\usepackage{authblk}
\usepackage{hyperref}

\newcommand{\R}{\mathbb{R}}
\newcommand{\N}{\mathbb{N}}

\newcommand{\bb}[1]{\mathbf{#1}}

\newcommand{\set}[1]{\left\lbrace #1  \right\rbrace}

\newcommand{\abs}[1]{\left| #1 \right|}

\newcommand{\br}[1]{\left( #1 \right)}
\newcommand{\sr}[1]{\left[ #1 \right]}

\newcommand{\sumx}[3]{\sum\limits_{#2}^{#3}{\vphantom\sum}_{\!#1}}

\DeclareMathOperator*{\union}{\bigcup}
\DeclareMathOperator*{\intersect}{\bigcap}

\DeclareMathOperator{\closure}{cl}

\newtheorem[style=L,cut=false]{theorem}{Theorem}
\newtheorem[style=L,cut=false]{corollary}{Corollary}
\newtheorem[style=L,cut=false]{hypothesis}{Hypothesis}
\newtheorem[style=L,cut=false]{lemma}{Lemma}
\newtheorem[style=L,bodystyle=\textnormal\noindent,cut=false]{definition}{Definition}
\newtheorem[style=L,bodystyle=\textnormal\noindent,cut=true]{remark}{Remark}
\newtheorem[style=L,bodystyle=\textnormal\noindent,cut=true]{conclusion}{Conclusion}
\newtheorem[style=L,bodystyle=\textnormal\noindent,cut=false]{notation}{Notation}

\providecommand{\keywords}[1]
{
  \small	
  \textbf{\textit{Keywords---}} #1
}

\title{A Topological View on Integration and Exterior Calculus}
\author{Petal B. Mokryn}
\affil{Monash University, Wellington Rd, Clayton VIC 3800, Australia}
\date{\today}

\begin{document}
\maketitle
\begin{abstract}
    A construction of integration, function calculus, and exterior calculus is made, allowing for integration of unital magma valued functions against (compactified) unital magma valued measures over arbitrary topological spaces. The Riemann integral, geometric product integral, and Lebesgue integral are shown as special cases. Notions similar to chain complexes are developed to allow this form of integration to define notions of exterior derivative for differential forms, and of derivatives of functions as well. Resulting conclusions on integration, orientation, dimension, and differentiation are discussed. Applications include calculus on fractals, stochastic analysis, discrete analysis, and other novel forms of calculus. 
\end{abstract}


\tableofcontents
\section{Introduction}
Analysis is and has been for a long time one of the most prominent fields studied in mathematics. As such, there have been many, many attempts at generalizing it to various settings, both specific and general. Ranging over discrete calculus\cite{dimakis1994discrete,fiadh2023applications,desbrun2006discrete,desbrun2005discrete}, graph analysis\cite{brouwer2011spectra}, calculus and differential forms on fractals via approaches such as harmonic analysis\cite{strichartz2006differential,kigami2001analysis,kelleher2017differential,ionescu2012derivations}, the fractal-fractional calculus\cite{yang2012advanced,kempfle2002fractional} and others\cite{stillinger1977axiomatic,heinonen2007nonsmooth,harrison2015operator}, as well as other novel forms of analysis such as the p-adic calculus\cite{robert2013course,albeverio2009operator}, stochastic and rough path analysis\cite{coutin2002stochastic,cont2013functional,stepanov2021towards,ito2015rough}, generalizations of analysis to more abstract algebraic structures\cite{bertram2008differential,dimakis1994differential,sardanashvily2016differential,bertram2003differential,bertram2020functorial}, noncommutative geometry\cite{dimakis1993noncommutative}, and more, inclusing works that attempt to unify various forms of calculus under a single framework\cite{harrison2015operator,hairer2014theory,brudnyi2015differential}. This is yet another such attempt. Inspired by the realization remarked on by David\cite{david2020laplacian} that the usual harmonic analysis approach to fractal calculus is near purely topological, the works of Harrison\cite{harrison2015operator} trying to extend the standard exterior calculus to apply on novel topological spaces embedded in $\R^{n}$, and the following line of thought: \textit{"Since a topological space is enough information to define measurable sets, shouldn't it be enough information to define integration? And if it is, would that be enough to define notions of differentiation via the generalized Stokes theorem, taken as a defining axiom?"}\\
The answer, as will be shown in this paper, seems to be mostly \textit{yes}. Section \ref{A Topological Construction of Integration} of this paper will define integration over measurable subset of an arbitrary topological space, of unital magma\footnote{A unital magma is a set closed under a binary operation, such that the set includes an identity element of said binary operation.} valued functions, against measures valued in compactifications of a unital magma that satisfy certain conditions. A form of integration I call \textit{IA integration} will be developed, and I will demonstrate the Riemann integral, geometric product integral, and Lebesgue integral on the real number line to be special cases of it. Section \ref{A Topological Construction of Integrable Chain Complexes} will construct notions of orientation and order of operations suited to the possibly noncommutative notions of integration considered in this paper, resulting in a structure I call an \textit{"integrable chain complex"}. Integrable chain complexes will be reminiscent of traditional chain complexes, but suited specifically for the form of calculus this paper defines. Section \ref{The Exterior Calculus on Integrable Chain Complexes} will develop the resulting basic notions of exterior calculus, including \textit{differential forms}, the \textit{exterior derivative}, and an \textit{integrable cochain complex} dual to the integrable chain complex on a topological space. Section \ref{The Function Calculus on Integrable Chain Complexes} will use a similar construction on integrable chain complexes to define the \textit{derivatives} of functions defined on certain subsets of the base topological space, as determined by the integrable chain complex considered. This will lead to a type of differential operator on functions with notion of dimension, replacing "functions on a $k$-dimensional surface" with "functions on the (closure of base sets of) integrable $k$-chains". Of particular note will be resulting conclusions on the nature of integration, orientation, a notion of integer dimension induced by the integrable chain complex on a topological space, and the behavior of noncommutative integrals. Sadly, due to time constraints, I have not yet been able to explore questions of what are the conditions for existence and uniqueness of many structures defined in this paper, see \hyperref[further work]{further work} for details. Although some questions of existence and uniqueness were not dealt with, some conclusions were reached on what properties these objects must have when well-defined. This paper doesn't proclaim to be contain a fully developed theory, but rather the skeleton of one - it gives the outline to defining a form of integration, the definition of an integrable chain complex, a resulting basic form of exterior calculus, and the definition of a resulting differential operator. Choosing a specific form of IA integral, constructing an appropriate integrable chain complex, and fully deriving the possible resulting differential operators is left for future work, and likely requires a case-by-case treatment for different forms of analysis defined on different topological spaces. Future works may be on fleshing this skeleton of a theory out more rigorously, constructing further notions of exterior and function calculus in its general setting, applying said the theory developed here as a general guideline to defining specific choices of calculus on specific topological spaces, and more.

\section{A Topological Construction of Integration}\label{A Topological Construction of Integration}
\subsection{Topological unital magmas and measures on the Borel $\sigma$-algebra of a topological space}

\subsubsection{Topological unital magmas}
First, I will define the algebraic structure that will be key to defining functions and measures to be integrated in the first place.
\begin{definition}
    A \textbf{topological unital magma} is a unital magma $(M,+_{M})$ that has an identity element $0_{M}$, where $M$ is endowed with a Hausdorff topology in which the magma operation $+_{M}:M\times M\to M$ is a continuous map.
\end{definition}
\begin{definition}
    An \textbf{extended} topological unital magma $(\overline{M},+_{M})$ is any compactification of a topological unital magma that is Hausdorff, where the magma operation $+_{M}$ is extended to be defined on all of $\overline{M}\times M$ and all of $M\times \overline{M}$ such that the identity element $0_{M}$ remains an identity element over all of $\overline{M}$.
\end{definition}
Note, in this definition the magma operation $+_{M}$ doesn't have to be defined on $(\overline{M}\setminus M)\times (\overline{M}\setminus M)$.
\begin{notation}
    The elements of $\overline{M}\setminus M$ are denoted the \textbf{infinities} of $M$.
\end{notation}
\begin{notation}\label{topology of a topological magma notation}
    The topology of a topological unital magma $(M,+_{M})$ is marked $\tau_{M}$ accordingly. The topology of the corresponding extended topological unital magma is marked $\tau_{\overline{M}}$.
\end{notation}
Recall, a net $x_{\bullet}:A\to \overline{M}$ converges to a limit $x^{*}\in \overline{M}$ if for any open neighborhood $U$ of $x^{*}$ in $\overline{M}$, there exists some $\alpha_{0}\in A$ such that for all $A\ni \alpha>\alpha_{0}$, $x_{\alpha}\in U$. Since an extended topological magma is by definition a compact space, any net $x_{\bullet}:A\to \overline{M}$ has a convergent subnet, and since the extended topological magma is required to be Hausdorff, then the limit of any net, if it exists, is unique in $\overline{M}$. If a net $x_{\bullet}:A\to \overline{M}$ converges to a limit point in $M$, it's said to \textbf{converge to a point} in $M$. If it converges to a limit point in $\overline{M}\setminus M$, the net is said to \textbf{converge to an infinity} of $M$.

\subsubsection{Measures on the Borel $\sigma$-algebra of a topological space}\label{Measures on the Borel sigma algebra}
Now I will define a sense of "measure", as general as I can make it, that allows for notions of integration to be considered. Given a topological space $X$, let $\Sigma_{X}$ be the Borel $\sigma$-algebra of $X$.
\begin{definition}
    A subset $U\subseteq X$ is called a \textbf{measurable subset} of $X$ if it is an element of the Borel $\sigma$-algebra $\Sigma_{X}$ of $X$.
\end{definition}
\begin{definition}\label{measure definition}
    An \textbf{orientable measure}, or just \textbf{measure} for short, $\mu$ on a topological space $X$ is a function ${\mu:\Sigma_{X}\to\overline{M}}$ from the Borel $\sigma$-algebra of $X$ to an extended topological unital magma  $(\overline{M},+_{M})$, such that $\mu(\emptyset)=\bb{0}_{M}$.
\end{definition}
In this context, $\overline{M}$ is called the \textbf{measure codomain}. This notion of measure is in fact, as will be seen later in this paper, a generalization not only of ordinary measures, but of signed measures too. The fact that signed measures are included as a special case of orientable measures will be needed to define integration on oriented sets. Given a measure $\mu$, a measurable subset $A\in\Sigma_{X}$ is said to be of \textbf{zero measure} with respect to $\mu$ if $\mu(A) = \bb{0}_{M}$, of \textbf{finite measure} with respect to $\mu$ if $\mu(A)\in M$, and of \textbf{infinite measure} with respect to $\mu$ if $\mu(A) \in\overline{M}\setminus M$.
\begin{definition}
    The collection of all possible measures over a topological space $X$ with measure codomain $\overline{M}$ is called the \textbf{space of measures} on $X$ with target $\overline{M}$, and denoted as $\Xi(X,\overline{M})$.
\end{definition}
\begin{corollary}
    The space of measures $\Xi(X,\overline{M})$ inherits the unital magma operation of $M$ via
    \begin{equation}
        \forall U\in\Sigma_{X},\forall \mu,\nu\in\Xi(X,\overline{M}): (\mu +_{M} \nu)(U) := \mu(U) +_{M} \nu(U)
    \end{equation}
    And inherits its identity element via
    \begin{equation}
        \bb{0}_{\Xi}:\Sigma_{X}\to\bb{0}_{M}
    \end{equation}
    So "addition of measures" is well-defined on all pairs of measurable subsets of $X$ that are not both of infinite measure. Instead of denoting $\bb{0}_{\Xi}$ and $\bb{0}_{M}$ separately, I will denote $\bb{0}_{M}$ for both.
\end{corollary}

\subsection{Functions and integration on the Borel $\sigma$-algebra of a topological space}
\subsubsection{Simple integrals}
\begin{notation}
    The collection of all functions $f:X\to Y$, with $(Y,+_{Y})$ a topological unital magma titled in this context as the \textbf{function codomain}, is to be called the \textbf{space of functions} on $X$ with target $Y$, and marked $\mathcal{F}(X,Y)$.
\end{notation}
\begin{corollary}
    The space of functions $\mathcal{F}(X,Y)$ inherits the unital magma operation of $Y$ via
    \begin{equation}
        \forall f,h\in\mathcal{F}(X,Y),\forall x\in X: \br{f+_{Y}h}(x) := f(x) +_{Y} h(x)
    \end{equation}
    and the identity element of $Y$ via
    \begin{equation}
        \bb{0}_{\mathcal{F}}:X\to\bb{0}_{Y}
    \end{equation}
    So "addition of functions" is well-defined on $X$. Instead of denoting $\bb{0}_{\mathcal{F}}$ and $\bb{0}_{Y}$ separately, I will denote $\bb{0}_{Y}$ for both.
\end{corollary}
\begin{definition}\label{integration element definition}
    Given a topological space $X$, a measure codomain $(\overline{M},+_{M})$, and a function codomain $(Y,+_{Y})$, an \textbf{integration element} is a mapping $g:Y\times\overline{M}\to\overline{G}$, such that $g$ is bi-distributive over the magma operations $+_{Y}$ and $+_{M}$:
    \begin{enumerate}
        \item $\forall a,b\in Y,m\in \overline{M}: g(a +_{Y} b,m) = g(a,m) +_{G} g(b,m)$

        \item $\forall a\in Y,m\in M,n\in\overline{M}: \begin{cases}
            g(a,m +_{M} n) = g(a,m) +_{G} g(a,n)\\
            g(a,n +_{M} m) = g(a,n) +_{G} g(a,m)
        \end{cases}$
    \end{enumerate}
    And such that
    $$\forall m\in \overline{M}: g\br{\bb{0}_{Y},m} = \bb{0}_{G}$$
    Where $(\overline{G},+_{G})$ is an extended topological unital magma titled the \textbf{integration codomain}.
\end{definition}
The integration codomain $(\overline{G},+_{G})$ is an extended topological unital magma that partially inherits its magma operation $+_{G}$ from $+_{Y}$ and $+_{M}$, and fully inherits its identity element via $\forall a\in Y,m\in\overline{M}: g(\bb{0}_{Y},m) = \bb{0}_{G}$.
\begin{definition}
    Given a mapping $f:A\times B\to C$ from two unital magmas $(A,+_{A}),(B,+_{B})$ to a unital magma $(C,+_{C})$ that is bi-distributive in the sense shown in definition \ref{integration element definition}, the \textbf{arithmetic structure} of $f$ is the expression of $f\br{a+_{A}b,c+_{B}d}$ in terms of $f(a,c),f(b,c),f(a,d),f(b,d)$ and $+_{C}$ for all $a,b\in A,c,d\in B$. Note, if the unital magma operations $+_{A},+_{B},+_{C}$ are commutative and associative, there is only one possible arithmetic structure of $f$. If $+_{A},+_{B},+_{C}$ aren't commutative and associative, there is more than one possible arithmetic structure, but the number of possibilities is finite.
\end{definition}
Note that integration elements $g:Y\times\overline{M}\to\overline{G}$ are such bi-distributive maps over $(Y,+_{Y}),(M,+_{M})$ and $(G,+_{G})$, so any integration element must have an arithmetic structure accordingly. Now that the integration element is defined, a notion of integration can be defined accordingly. Let $X$ be a topological space with Borel $\sigma$-algebra $\Sigma_{X}$, let $(Y,+_{Y})$ be a function codomain, let $\Xi(X,\overline{M})$ be a space of measures on $X$ with measure codomain $(\overline{M},+_{M})$, and let $g:Y\times\overline{M}\to \overline{G}$ be the element of integration with integration codomain $(\overline{G},+_{G})$.
\begin{definition}
    Given the above setting, an \textbf{indicator function} of a measurable subset $S\in\Sigma_{X}$ is any function $I_{S}\in\mathcal{F}(X,Y)$ defined as follows:
    \begin{equation}\label{indicator function}
        \forall x\in X: I_{S}(c:x) = \begin{cases}
            c &, x\in S\\
            0_{Y} &, x\notin S
        \end{cases}
    \end{equation}
    Where $0_{Y}\neq c\in Y$ is some non-identity constant in a unital magma $(Y,+_{Y})$.
\end{definition}
\begin{definition}\label{definition of simple integral}
    Given the above setting, a \textbf{simple integral} $\int\limits_{U}g\br{f,d\mu}$ of a function $f:X\to Y$ against a measure $\mu:\Sigma_{X}\to\overline{M}$ over a measurable subset $U\in\Sigma_{X}$ is any map $\int\limits_{U}:\mathcal{F}(X,Y)\times\Xi(X,\overline{M})\to \overline{G}$ that satisfies the following properties:
    \begin{enumerate}
        \item \textbf{Arithmetic structure:}\\
        The simple integral operation $\int\limits_{U}:\mathcal{F}(X,Y)\times\Xi(X,\overline{M})\to \overline{G}$ must be a bi-distributive mapping sharing the same arithmetic structure that the integration element $g:Y\times\overline{M}\to \overline{G}$ has.

        \item \textbf{Simple integration of indicator functions:}\\
        Given an indicator function $I_{S}(c:x) = \begin{cases}
            c &, x\in S\\
            0_{Y} &, x\notin S
        \end{cases}$ with $U,S\in\Sigma_{X}$ measurable subsets of $X$, and given a measure $\mu\in\Xi(X,\overline{M})$, the simple integral of $I_{S}(c:x)$ is
        \begin{equation}\label{constant functions are indicator functions}
            \int\limits_{U}g\br{I_{S}(c:x),d\mu}:=g\br{c,\mu\br{U\intersect S}}
        \end{equation}
    \end{enumerate}
\end{definition}
Note, the simple integral doesn't have to be defined on all of $\mathcal{F}(X,Y)\times\Xi(X,\overline{M})$, only on indicator functions and on \textit{integrable simple functions}, see definitions \ref{simple function definition}, \ref{integrable simple functions} and theorem \ref{theorem: integral of a simple function}.\\
Example of the arithmetic structure demand of definition \ref{definition of simple integral}:\\
\textbf{If} $\forall a,b\in Y,m,n\in M:g(a+_{Y}b,m+_{M}n) = g(a,m) +_{G} g(b,m) +_{G} g(a,n) +_{G} g(b,n)$ \textbf{then}\\
\footnotesize
$\forall f,h\in\mathcal{F}(X,Y),\mu,\nu\in\Xi(X,\overline{M}):\int\limits_{U}g\br{f +_{\mathcal{F}} h\,,\,d\mu +_{\Xi} d\nu} = \int\limits_{U}g\br{f,d\mu} \,+_{G}\,\int\limits_{U}g\br{h,d\mu} \,+_{G}\, \int\limits_{U}g\br{f,d\nu} \,+_{G}\,  \int\limits_{U}g\br{h,d\nu}$\mbox{}\\
And\\
\textbf{If} $\forall a,b\in Y,m,n\in M:g(a+_{Y}b,m+_{M}n) = g(a,m) +_{G} g(a,n) +_{G} g(b,m) +_{G} g(b,n)$ \textbf{then}\\
\footnotesize
$\forall f,h\in\mathcal{F}(X,Y),\mu,\nu\in\Xi(X,\overline{M}):\int\limits_{U}g\br{f +_{\mathcal{F}} h\,,\,d\mu +_{\Xi} d\nu} = \int\limits_{U}g\br{f,d\mu} \,+_{G}\, \int\limits_{U}g\br{f,d\nu} \,+_{G}\,\int\limits_{U}g\br{h,d\mu} \,+_{G}\,  \int\limits_{U}g\br{h,d\nu}$\mbox{}\\
\normalsize
With the relations holding whenever the RHS of 4 equations above is well-defined. In short, going from integration element to simple integral is a structure-preserving process, similar to the notion of homomorphism.
\begin{corollary}\label{bi-distributivity of integrals}
    A simple integral $\int\limits_{U}:\mathcal{F}(X,Y)\times\Xi(X,\overline{M})\to\overline{G}$ is by definition bi-distributive over functions in $\mathcal{F}(X,Y)$ and measures in $\Xi(X,\overline{M})$ the same way its integration element $g:Y\times\overline{M}\to\overline{G}$ is bi-distributive over elements of $Y$ and $\overline{M}$:
    \begin{align}
        \forall f,h\in\mathcal{F}(X,Y),\mu\in\Xi(X,\overline{M}): \int\limits_{U}g(f+_{Y}h,\mu) &= \int\limits_{U}g(f,\mu) +_{G} \int\limits_{U}g(h,\mu)\\
        \forall f\in\mathcal{F}(X,Y),\mu,\nu\in\Xi(X,\overline{M}): \int\limits_{U}g(f,\mu+_{M}\nu) &= \int\limits_{U}g(f,\mu) +_{G} \int\limits_{U}g(f,\nu)
    \end{align}
\end{corollary}

\subsubsection{Simple functions and IA integration}
\begin{notation}
    Given a unital magma $(A,+_{A})$ and a sequence of elements $a_{1},...,a_{n}\in A$, the following notation is used:
    \begin{equation}
        \sumx{A}{i=1}{n}a_{i} := a_{1} \,+_{A}\,...\,+_{A}\, a_{n}
    \end{equation}
    Where the order of operations is chosen to be from left to right. An analogous theory exists for an order of operations chosen to be from right to left, and for any other choice of ordering - the choice need only be consistent.
\end{notation}
\begin{definition}\label{simple function definition}
    A \textbf{simple function} on the topological space $X$ is a finite magma sum of indicator functions of measurable subsets of $X$:
    \begin{equation}
        \forall x\in X: f_{n}(x) := \sumx{Y}{k=1}{n}I_{S_{k}}(c_{k}:x)
    \end{equation}
    Where $\forall k=1,...,n: S_{k}\in\Sigma_{X}$.
\end{definition}
\begin{definition}\label{integrable simple functions}
    Given a measure $\mu\in\Xi(X,\overline{M})$ on $X$, a simple function $f_{n}(x) = \sumx{Y}{k=1}{n}I_{S_{k}}(c_{k}:x)$ is said to be \textbf{integrable} with respect to $\mu$ on a measurable subset $U\in\Sigma_{X}$ if
    \begin{equation}
        \forall k=1,...,n: \mu\br{S_{K}\intersect U} \in M
    \end{equation}
\end{definition}
\begin{theorem}\label{theorem: integral of a simple function}
    \textbf{Simple integration of simple functions:} It trivially follows from definition \ref{definition of simple integral} that given a simple function $f_{n}(x) = \sumx{Y}{k=1}{n}I_{S_{k}}(c_{k}:x)$ that is integrable with respect to a measure $\mu\in\Xi(X,\overline{M})$ on a measurable subset $U\in\Sigma_{X}$, the simple integral of $f_{n}$ against $\mu$ on $U$ must be
    \begin{equation}\label{the integral of a simple function}
        \int\limits_{U}g\br{f_{n},d\mu} = \sumx{G}{k=1}{n}\, g\br{c_{k},\mu\br{S_{k}\intersect U}}
    \end{equation}
    And definition \ref{integrable simple functions} guarantees that the unital magma sum on the RHS of eq. \ref{the integral of a simple function} is a well-defined element of the integration codomain $\overline{G}$.
\end{theorem}
Now that I have a notion for integration of simple functions, I can use it to define integration of non-simple functions as well. As per notation \ref{topology of a topological magma notation}, let $\tau_{Y}$ be the topology of the topological unital magma $(Y,+_{Y})$.
\begin{definition}
    Given a function $f\in\mathcal{F}(X,Y)$ and a measure $\mu\in\Xi(X,\mathcal{M})$, an \textbf{integrable approximation} (IA) of $f$ with respect to $\mu$ on $U$, denoted $\set{s_{k}|U}_{k=1}^{\infty}\stackrel{\mu}{\to} f$, on a measurable subset $U\in\Sigma_{X}$ is a series of simple functions $\set{s_{k}}_{k=1}^{\infty}$ such that:
    \begin{enumerate}
        \item For all $k\in\N$, $s_{k}$ is integrable on $U$ with respect to $\mu$
        
        \item The sequence $\set{s_{k}}_{k=1}^{\infty}$ converges pointwise to $f$ for all $x\in U$:\\
        For all $x\in U$, and for any open neighborhood $V\in\tau_{Y}$ such that $f(x)\in V$, there exists some $K\in\N$, such that for all $k>K$, $s_{k}(x)\in V$.
    \end{enumerate}
\end{definition}
For any IA $\set{s_{k}|U}_{k=1}^{\infty}\stackrel{\mu}{\to} f$ of a function $f\in\mathcal{F}(X,Y)$ on $U$ with respect to a measure $\mu\in\Xi(X,\mathcal{M})$, there is an associated sequence of simple integrals $\set{\int\limits_{U}g\br{s_{k},d\mu}}_{k=1}^{\infty}$ that by theorem \ref{theorem: integral of a simple function} are all well-defined. Therefore, I can now define:
\begin{definition}\label{IA integral definition}
    Given a function $f\in\mathcal{F}(X,Y)$ and a collection of infinite sequences of simple functions\\
    $L[f](X,Y)\subseteq \br{\mathcal{F}(X,Y)}^{\N}$, the \textbf{$L[f](X,Y)$-IA integral} of $f$ on $U$ with respect to a measure\\
    $\mu\in\Xi(X,\mathcal{M})$, is, if it exists, the unique value $I\in \overline{G}$ of the integration codomain $\overline{G}$ such that $$I = \lim\limits_{k\to\infty}\int\limits_{U}g\br{s_{k},d\mu}$$ for all possible IAs $\set{s_{k}|U}_{k=1}^{\infty}\stackrel{\mu}{\to} f$ such that $\set{s_{k}}_{n=1}^{\infty}\in L[f](X,Y)$.\\
    The $L[f](X,Y)$-IA integral of a function $f$ on a region $U$ against measure $\mu$ is simply denoted as $\int\limits_{U}g(f,d\mu)$, the same as in definition \ref{definition of simple integral}.
\end{definition}
In the context of definition \ref{IA integral definition}, an IA $\set{s_{k}|U}_{k=1}^{\infty}\stackrel{\mu}{\to} f$ such that $\set{s_{k}}_{n=1}^{\infty}\in L[f](X,Y)$ is called an $L[f](X,Y)$-IA of $f$.
\begin{definition}
    If a function $f\in\mathcal{F}(X,Y)$ has any two $L[f](X,Y)$-IAs, $\set{s_{1,k}|U}_{k=1}^{\infty}\stackrel{\mu}{\to} f$ and $\set{s_{2,k}|U}_{k=1}^{\infty}\stackrel{\mu}{\to} f$ such that $\lim\limits_{k\to\infty}\int\limits_{U}g\br{s_{1,k}d\mu} \neq \lim\limits_{k\to\infty}\int\limits_{U}g\br{s_{2,k}d\mu}$, then the $L[f](X,Y)$-IA integral of $f$ on $U$ with respect to $\mu$ is said to be \textbf{undefined}.
\end{definition}
For functions $f\in\mathcal{F}(X,Y)$ whose $L[f](X,Y)$-IA integral is undefined on a region $U\in\Sigma_{X}$ with respect to a measure $\mu\in\Xi(X,\overline{M})$, a construction analogous to Cauchy's principal value may be achieved by defining a systematic way to choose a specific reduced collection $\mathcal{L}[f](X,Y)\subset L[f](X,Y)$ on $U$, such that the $\mathcal{L}[f](X,Y)$-IA integral of $f$ is (hopefully) defined. The specific choice of $\mathcal{L}[f](X,Y)$ would then fulfill the same role as the classic Cauchy principal value (or any of its variations) does in ordinary analysis. However, there may be some cases where even that can't be done:
\begin{definition}
    Given a function $f\in\mathcal{F}(X,Y)$, if there doesn't exist any $L[f](X,Y)$-IA of $f$ on $U$ with respect to $\mu$ that has a convergent integral series, then the $L[f](X,Y)$-IA integral of $f$ on $U$ with respect to $\mu$ is said to be \textbf{fundamentally undefined}.
\end{definition}
Note, that while by theorem \ref{theorem: integral of a simple function} the simple integral of an integrable simple function $s$ is always defined, the question of whether or not its $L[s](X,Y)$-IA integral is defined depends on the choice of $L[s](X,Y)$.\\
\noindent\rule{12cm}{0.4pt}
\begin{example}
    \textbf{IA-Riemann integration}\\
    Let the topological space in question be $X=\R$, let the function codomain be $(Y,+_{Y})=(\R,+)$, let the measure codomain be $(\overline{M},+_{M})=([0,\infty],+)$, and let the element of integration be
    $$\forall a\in Y,m\in \overline{M}: g(a,m) = a\cdot m$$
    With integration codomain $(\overline{G},+_{G})=\br{[-\infty,\infty],+}$, with $+,\cdot$ being ordinary real number addition and multiplication. For an interval $[a,b]\subset\R$, a \textbf{tagged partition} of $[a,b]$ is a structure of the form $P = \set{t_{i},[x_{i},x_{i+1}]}_{i=1}^{n}$ with $n\in\N$, such that $t_{i}\in[x_{i},x_{i+1}]$ and $x_{i+1}>x_{i}$ for all $i=1,...,n$, and $[a,b]=\union\limits_{i=1}^{n}[x_{i},x_{i+1}]$. Given a function $f\in\mathcal{F}(X,Y)$, let the \textbf{$P$-approximation} of $f$ on $[a,b]$ be the simple function on $[a,b]$ defined as $\forall x\in[a,b]:f_{P}(x) = \sum\limits_{i=1}^{n}I_{[x_{i},x_{i+1}]}(f(t_{i});x)$.\\
    The \textbf{IA-Riemann integral}, of a function $f$ on $[a,b]$ against a measure $\mu\in\Xi([a,b],[0,\infty])$ is the $L[f]([a,b],\R)$-IA integral of $f$ on $[a,b]$ against $\mu$, with $L[f]([a,b],\R)$ chosen such that only IAs made of $P$-approximations of $f$ on $[a,b]$, for all partitions $P$ of $[a,b]$, are allowed. The IA-Riemann integral clearly coincides with the ordinary Riemann integral, as partition refinements define an IA for all Riemann-integrable functions on $[a,b]$, and it's easy to show that a function is Riemann-integrable in the ordinary sense on $[a,b]$ if and only if it's IA-Riemann integrable on $[a,b]$. Definition \ref{IA integral definition} naturally gives that the IA-Riemann integral is to be denoted as $\int\limits_{[a,b]}fd\mu$, accordingly. The domain of this integral can be generalized from an interval $[a,b]$ to arbitrary unions of such intervals.
\end{example}
\begin{example}
    \textbf{The IA geometric product integral}\\
    Let the topological space in question be $X=\R$, let the function codomain be $(Y,+_{Y})=(\R^{+},\cdot)$, let the measure codomain be $(\overline{M},+_{M})=([0,\infty],+)$, and let the element of integration be
    $$\forall a\in Y,m\in \overline{M}: g(a,m) = a^{m}$$
    With integration codomain $(\overline{G},+_{G})=\br{[0,\infty],\cdot}$, with $+,\cdot$ being ordinary real number addition and multiplication, and $a^{m}$ being $a$ to the power of $m$. Defining partitions of the interval $[a,b]$ the same way as in the previous example, given a function $f\in\mathcal{F}(X,Y)$, let the \textbf{geometric $P$-approximation} of $f$ on $[a,b]$ be the simple function on $[a,b]$ defined as $\forall x\in[a,b]:f_{P}(x) = \prod\limits_{i=1}^{n}I_{[x_{i},x_{i+1}]}(f(t_{i});x)$. Note, since the function codomain's magma operation $+_{Y}$ is in this case real number multiplication, the identity element used in the definition  of indicator functions is $\bb{0}_{Y}=1$.  The \textbf{IA geometric product integral} of a function $f$ on $[a,b]$ against a measure $\mu\in\Xi([a,b],[0,\infty])$ is the $L[f]([a,b],\R)$-IA integral of $f$ on $[a,b]$ against $\mu$, with $L[f]([a,b],\R)$ chosen such that only IAs made of $P$-approximations of $f$ on $[a,b]$, for all partitions $P$ of $[a,b]$, are allowed. Definition \ref{IA integral definition} naturally gives that the IA geometric product integral\\
    is to be denoted as $\int\limits_{[a,b]}f^{d\mu}$, accordingly. The domain of this integral can be generalized from an interval $[a,b]$ to arbitrary unions of such intervals.
\end{example}
\begin{example}
    \textbf{The IA Lebesgue integral}\\
    Let the topological space in question be $X=\R$, let the function codomain be $(Y,+_{Y})=(\R,+)$, let the measure codomain be $(\overline{M},+_{M})=([0,\infty],+)$, and let the element of integration be
    $$\forall a\in Y,m\in \overline{M}: g(a,m) = a\cdot m$$
    With integration codomain $(\overline{G},+_{G})=\br{[-\infty,\infty],+}$, with $+,\cdot$ being ordinary real number addition and multiplication. The \textbf{IA Lebesgue integral} of a function $f\in\mathcal{F}\br{\R,\R}$ on $\R$ against a measure $\mu\in\Xi(\R,[0,\infty])$ is the $L[f](\R,\R)$-IA integral of $f$, with $L[f](\R,\R)$ chosen to only allow IAs $\set{s_{k}|U}_{k=1}^{\infty}\stackrel{\mu}{\to} f$ that satisfy the following conditions:
    \begin{enumerate}
        \item $\forall k\in\N, x\in\R: \abs{s_{k}(x)} \leq \abs{f(x)}$

        \item $\forall k\in\N, x\in\R: \abs{s_{k}(x)} \leq \abs{s_{k+1}(x)}$
    \end{enumerate}
    IAs $\set{s_{k}|U}_{k=1}^{\infty}\stackrel{\mu}{\to} f$ that satisfy these conditions are to be known as \textbf{Lebesgue IAs}. Definition \ref{IA integral definition} naturally gives that the IA Lebesgue integral above is to be denoted as $\int\limits_{\R}fd\mu$, accordingly.\\
    \textbf{Claim:} The IA Lebesgue integral is, in fact, the ordinary Lebesgue integral.
    \begin{proof}
        Let $f(x) = f^{+}(x) - f^{-}(x)$, where $f^{+}(x) :=\begin{cases}
            f(x) &f(x)\geq 0\\
            0 &f(x)\leq 0
        \end{cases}$ and $f^{-}(x):=\begin{cases}
            -f(x) &f(x)\leq 0\\
            0 &f(x)\geq 0
        \end{cases}$\\
        It's clear by corollary \ref{bi-distributivity of integrals} and the definition of IA integration in terms of simple integrals that for any measure $\mu\in\Xi(\R,[0,\infty])$, $\int\limits_{\R}fd\mu = \int\limits_{\R}f^{+}d\mu - \int\limits_{\R}f^{-}d\mu$. First looking at $\int\limits_{\R}f^{+}d\mu$, note that the simple integrals of simple functions coincide in this setting with the ordinary Lebesgue integration of simple functions. Thus, the monotone convergence theorem for non-negative functions can be applied, and it's clear that any Lebesgue IA $\set{s_{k}^{+}|U}_{k=1}^{\infty}\stackrel{\mu}{\to} f^{+}$ converges to a limit in $[0,\infty]$. Furthermore, the monotone convergence theorem implies that $\lim\limits_{k\to\infty}\int\limits_{\R}s_{k}^{+}d\mu = \int\limits_{\R}\lim\limits_{k\to\infty}s_{k}^{+}d\mu$, so therefore for any two Lebesgue IAs $\set{s_{1,k}^{+}|U}_{k=1}^{\infty}\stackrel{\mu}{\to} f^{+}$, $\set{s_{2,k}^{+}|U}_{k=1}^{\infty}\stackrel{\mu}{\to} f^{+}$, it results that $\lim\limits_{k\to\infty}\int\limits_{\R}s_{1,k}^{+}d\mu - \lim\limits_{k\to\infty}\int\limits_{\R}s_{2,k}^{+}d\mu = \int\limits_{\R}\lim\limits_{k\to\infty}\br{s_{1,k}^{+} - s_{2,k}^{+}}d\mu = \int\limits_{\R}0\cdot d\mu = 0$,\\
        so all Lebesgue IAs $\set{s_{k}^{+}|U}_{k=1}^{\infty}\stackrel{\mu}{\to} f^{+}$ converge to the same limit, implying that they all converge to the supremum $\sup\int\limits_{\R}sd\mu$ over all simple functions $s$ satisfying $\forall x\in\R: 0 \leq s(x) \leq f^{+}(x)$, meaning that the IA Lebesgue integral of the non-negative function $f^{+}$ equals its standard Lebesgue integral. Repeating the same process with $f^{-}$, and noting that for all Lebesgue IAs $\set{s_{k}^{+}|U}_{k=1}^{\infty}\stackrel{\mu}{\to} f^{+}\,,\,\set{s_{k}^{-}|U}_{k=1}^{\infty}\stackrel{\mu}{\to} f^{-}$, the sequence $\set{s_{k}}_{k=1}^{\infty}$ defined by $s_{k}(x) := \begin{cases}
            s_{k}^{+} &s_{k}^{+} \geq 0\\
            -s_{k}^{-} &s_{k}^{-} \geq 0
        \end{cases}$ is a Lebesgue IA of $f(x)$, and in fact, and Lebesgue IA of $f$ can be written that way. Therefore, the IA Lebesgue integral of $f$ is exactly equal to the subtraction of standard Lebesgue integrals $\int\limits_{\R}f^{+}d\mu - \int\limits_{\R}f^{-}d\mu$, thus making the IA Lebesgue integral of $f$ simply its ordinary Lebesgue integral.
    \end{proof}
\end{example}
\noindent\rule{12cm}{0.4pt}\\
Further examples such Riemann integration on PCF fractals a-la Strichartz\cite{strichartz2006differential}, stochastic It\^o integration, and more, can similarly be defined as special cases of IA integration. Note that while both of the above examples used measure codomains corresponding to unsigned measures for simplicity's sake, as will be needed in the next section, measures in this paper are generally meant as a generalization of signed measures.\\
\textbf{From this point on in the paper, only IA integrals will be considered, as their construction naturally avoids problems with integration on sets of infinite measure.}

\section{A Topological Construction of Integrable Chain Complexes}\label{A Topological Construction of Integrable Chain Complexes}
\subsection{Measures, functions, orientation, and IA integration on integrable chain spaces of a topological space}
\subsubsection{Measures and IA integration on subspace topologies}\label{Measures and integration on subspace topologies}
Given a topological space $X$, sometimes it makes sense to define a non-trivial "size", aka measure, on subsets of $X$ that are either not members of the Borel $\sigma$-algebra $\Sigma_{X}$, or may have zero measure for all non-singular measures on $\Sigma_{X}$. For example, the Lebesgue measure of a smooth 1D curve $c$ embedded in $\R^{n}$ is zero under the standard $\R^{n}$ topology, but one may still want a non-trivial "length" measure of the curve. For this purpose, I will now define:
\begin{definition}\label{subset Borel sigma-algebra}
    Given a topological space $X$ and a subset $S\subseteq X$, the \textbf{subspace Borel $\sigma$-algebra} of $S$ is the Borel $\sigma$-algebra $\Sigma_{S}$ generated by the subspace topology $\tau_{S}^{X}$ induced on $S$ by $X$.
\end{definition}
From there, all previous definitions and resulting properties of measures can be readily applied, simply replacing $\Sigma_{X}$ with $\Sigma_{S}$. As such, I will simply denote: 
\begin{notation}
    For a subset $S\subseteq X$ of a topological space $X$, the space of measures on the subspace Borel $\sigma$-algebra of $S$ is denoted as $\Xi(\Sigma_{S},\overline{M}_{S})$. 
\end{notation}
For example, for a 1D curve $c$ embedded in $\R^{n}$, the Lebesgue measure under the standard $\R^{n}$ topology will give $\mu_{\R^{n}}(c)=0$, but a suitable Lebesgue measure under the subspace topology of $c$ in $\R^{n}$ will give the arclength of $c$.
\begin{remark}
    If a subset $U\subseteq X$ is an open set of $\tau_{X}$, then $\Sigma_{U} = \set{U\intersect S\Big|\,S\in\Sigma_{X}}$. In such a case, the subspace Borel $\sigma$-algebra of $U$ is simply the restriction of $\Sigma_{X}$ to $U$. Thus, subsets $S\subseteq X$ with non-empty topological interior in $X$ can have non-zero continuous measure under $\Sigma_{X}$ as is, and using $\Sigma_{S}$ should generally only be necessary for subsets $S\subseteq X$ whose interior in $X$ is empty.
\end{remark}
Going on from defining measures on subspace topologies, integration of a function $f:S\to Y$ on a subset $S\subseteq X$ can be defined relative to the subspace Borel $\sigma$-algebra $\Sigma_{S}$ instead of $\Sigma_{X}$, and all the above definitions of general integration maps and IA integrations then readily apply.

\subsubsection{Integrable chain spaces}
For the domains of integration on a topological space, one may be tempted to just use subsets of the the space in question, but such a construction lacks much of the necessary information. First, besides the sets to integrate on, a notion of the orientation of the sets to integrate over is needed - e.g. when integrating around a curve, is one integrating clockwise, or counterclockwise? For this reason, chain complexes are often used to define integration, but they may not always be sufficiently general. What if the "addition" operation underlying the chosen notions of integration is noncommutative, or even nonassociative? What if not just orientation matters, but also the \textit{order of operations}? Hence, a construction similar to but not the same as chain complexes is made to accommodate such cases, to encode the information not just of orientation, but also of the ordering.\\
\paragraph{Integrable chains}\mbox{}\\
First, I need to consider the collection of subsets integration is carried over:
\begin{definition}\label{integrable collection definition}
    Given a topological space $X$, an \textbf{integrable collection} on $X$ is a collection $B\subseteq\mathcal{P}(X)$ of subsets of $X$, where $\mathcal{P}(X)$ is the power set of $X$, such that:
    \begin{enumerate}
        \item \textbf{Closure under Borel $\sigma$-algebra:} If $S$ is a member of $B$, then all measurable subsets $V\in\Sigma_{S}$ are also members of $B$, where $\Sigma_{S}$ is as defined in definition \ref{subset Borel sigma-algebra}.

        \item \textbf{Closure under finite unions:} If $U$ and $V$ are members of $B$, then $U\union V$ is a member of $B$.
    \end{enumerate}
\end{definition}
\begin{notation}
    Note, integrable collections on a topological space $x$ can be easily generated by taking any collection of $A$ subsets of $X$, taking all of their finite unions, and then taking the subspace Borel $\sigma$-algebra on each of resulting subsets of $X$ relative to $X$. In this setting, $A$ is to be called a \textbf{generating collection} of the integrable collection $B$.
\end{notation}
For example, the collection of all subsets of a topological space $X$ that are locally homeomorphic to the unit interval $(0,1)$, corresponding to the set of all smooth 1D paths through $X$, is a generating collection of an integrable collection on $X$, allowing to define line integrals on $X$.
\begin{definition}\label{integrable basis definition}
    Given a topological space $X$, an integrable collection $B$ on $X$, and a non-empty set $\mathcal{O}$, the corresponding \textbf{integrable basis} is the set $\mathcal{B}=\mathcal{O}\times B$, endowed with a function codomain $(Y,+_{Y})$, a measure codomain $(\overline{M},+_{M})$, an integration codomain $(\overline{G},+_{G})$, and a corresponding notion of IA integration $\int\limits_{S}:\mathcal{F}(S,Y)\times\Xi(S,\overline{M})\to\overline{G}$ with integration element $g:Y\times\overline{M}\to\overline{G}$ on all basic sets $S\in B$.
    \begin{enumerate}
        \item The elements of $B$ are called the \textbf{basic sets} of $\mathcal{B}$.

        \item The elements of $\mathcal{O}$ are called the \textbf{orientations} of $\mathcal{B}$.

        \item The elements of the $\mathcal{B}$ are called the \textbf{oriented sets} of $\mathcal{B}$ and denoted $\iota_{k}S$, where $\iota_{k}\in\mathcal{O}$ and $S\in B$.
    \end{enumerate}
    Each orientation $\iota_{k}\in\mathcal{O}$ must correspond to a function $\iota_{k}:M\to M$ that is continuous in the unital magma topology of $M$, and there must be a unique orientation symbol $e\in\mathcal{O}$ that corresponds to the identity map on $M$ via $\forall m\in M: e(m) := \bb{0}_{M} +_{M} m = m +_{M} \bb{0}_{M} = m$.
\end{definition}
\begin{notation}
    In definition \ref{integrable basis definition}, the specific choice of orientations $\mathcal{O}$, function codomain $(Y,+_{Y})$,\\
    measure codomain $(\overline{M},+_{M})$, integration codomain $(\overline{G},+_{G})$, and IA integration\\
    $\int\limits_{S}:\mathcal{F}(S,Y)\times\Xi(S,\overline{M})\to\overline{G}$ with integration element $g:Y\times\overline{M}\to\overline{G}$ is called the \textbf{choice of calculus} forming $\mathcal{B}$, and denoted by $(\mathcal{O},Y,\overline{M},\overline{G},\int g)$. The integrable basis $\mathcal{B}$ is said to be \textbf{formed} on a topological space $X$ by the integrable collection $B$ and the choice of calculus\\
    $(\mathcal{O},Y,\overline{M},\overline{G},\int g)$.
\end{notation}
Now I can form domains of integration from the integrable basis, in a similar fashion to how chains are formed from simplexes in singular homology:
\begin{definition}\label{B-integrable chain definition}
    Given a topological space $X$ and an integrable basis $\mathcal{B}$ on $X$, a \textbf{$\mathcal{B}$-integrable chain} is a formal magma sum of a finite number of oriented sets:
    \begin{equation}
        c := \sumx{\mathcal{B}}{k=1}{n}\iota_{k}S_{k}
    \end{equation}
    Where each $\iota_{k}S_{k}\in\mathcal{B}$ is an oriented set, with each $\iota_{k}$ an orientation and each $S_{k}$ a basic set. The set $S_{c} := \union\limits_{k=1}^{n}S_{k}$ is called the \textbf{base set} of $c$. The magma operation $+_{\mathcal{B}}$ is defined as a purely formal operation, similar to formal group sums but not satisfying any properties other than being a binary operation on oriented sets.
\end{definition}
Note, definition \ref{integrable collection definition} guarantees that the base set $S_{c}$ of a $\mathcal{B}$-integrable chain $c$ is always a member of the integrable collection $B$.
\begin{notation}\label{chains over the same domain}
    Given an integrable basis $\mathcal{B}$ formed on a topological space $X$ by an integrable collection $B$ and a choice of calculus $(\mathcal{O},Y,\overline{M},\overline{G},\int g)$, two $\mathcal{B}$-integrable chains $c_{1},c_{2}$ are said to be \textbf{over the same domain} and denoted $c_{1}\stackrel{S}{\sim}c_{2}$ if their base sets are the same:
    $$c_{1}\stackrel{S}{\sim}c_{2} \Leftrightarrow S_{c_{1}}=S_{c_{2}}$$
\end{notation}

\paragraph{Measure, orientation and IA integration on integrable chains}\mbox{}\\
\begin{definition}
    Given an integrable basis $\mathcal{B}$ formed on a topological space $X$ by an integrable collection $B$ and a choice of calculus $(\mathcal{O},Y,\overline{M},\overline{G},\int g)$, the \textbf{space of measures} on $\mathcal{B}$ is the bundle defined as follows:
    \begin{equation}
        \Xi(\mathcal{B}) := \set{\br{S,\Xi(S,\overline{M})}\Big| S\in B}
    \end{equation}
    Measures on a specific set $S\in B$ will be denoted either as $\mu\in\Xi(S,\overline{M})$ or $(S,\mu)\in\br{S,\Xi(S,\overline{M})}$, interchangeably.
\end{definition}
\begin{definition}\label{oriented measure}
    Given an integrable basis $\mathcal{B}$ formed on a topological space $X$ by an integrable collection $B$ and a choice of calculus $(\mathcal{O},Y,\overline{M},\overline{G},\int g)$, an oriented set $\iota_{k}S\in\mathcal{B}$, and a measure $\mu\in\Xi(S,\overline{M})$, the corresponding \textbf{oriented measure} on $\iota_{k}S$ is the measure $\iota_{k}\mu\in\Xi(S,\overline{M})$ defined by
    \begin{equation}
        \forall A\in\Sigma_{S}:\iota_{k}\mu(A) := \br{\iota_{k}\circ\mu}(A)
    \end{equation}
\end{definition}
Note that $\iota_{k}\mu(A)$ may be undefined if $A$ is a set of infinite measure with respect to $\mu$, but that doesn't matter in the context of IA-integration against $\mu$, because the integral is computed as a limit of integrals of integrable simple functions, which by definition ensures only sets of finite measure with respect to $\mu$, and thus also finite with respect to $\iota_{k}\mu$, are considered in the first place.
\begin{remark}[A conclusion on the nature of orientation]\label{orientations remark}
    If the measure codomain $(\overline{M},+_{M})$ is such that $(M,+_{M})$ is a topological group, it makes most sense to have 2 orientation symbols: $\mathcal{O}=\set{e,-_{M}}$, corresponding to the identity map on $M$ and the group inverse map on $M$. \textit{This can be seen as the reason why there are only two possible orientations of domains of integration in ordinary exterior calculus.} Instead of being an underlying property of the topological space $X$, the possible orientations of subsets of $X$ can rather be seen as a property of the choice of calculus. Specifically, of the part (signed) measures play in defining the choice of calculus over an integrable collection of $X$.
\end{remark}
The above definitions now allow me to define integration on an integrable chain, of a function on its base set against a measure on its base set:
\begin{definition}\label{integration on chains}
    Given a $\mathcal{B}$-integrable chain $c = \sumx{\mathcal{B}}{k=1}{n}\iota_{k}S_{k}$ of an integrable basis $\mathcal{B}$ formed by an integrable collection $B$ and a choice of calculus $(\mathcal{O},Y,\overline{M},\overline{G},\int g)$ on a topological space $X$, the \textbf{integral} on $c$ of a function $f\in\mathcal{F}\br{S_{c}, Y}$ against a measure $\mu\in\Xi(S_{c},\overline{M})$ is defined as follows:
    \begin{equation}
        \int\limits_{c}g(f,d\mu) := \sumx{G}{k=1}{n}\int\limits_{S_{k}}g(f,\iota_{k}d\mu)
    \end{equation}
    Where each integral $\int\limits_{S_{k}}g(f,\iota_{k}d\mu)$ is carried out with respect to function $f$ restricted to $S_{k}$, and the measure $\iota_{k}\mu$ (see definition \ref{oriented measure}) restricted to $\Sigma_{S_{k}}$.
\end{definition}

\paragraph{Integrable chain spaces}\mbox{}\\
\begin{definition}\label{integration equivalence definition}
    Given an integrable basis $\mathcal{B}$ formed by an integrable collection $B$ and a choice of calculus\\ $(\mathcal{O},Y,\overline{M},\overline{G},\int g)$ on a topological space $X$, two $\mathcal{B}$-integrable chains $c_{1},c_{2}$ are said to be \textbf{integration equivalent} if:
    \begin{enumerate}
        \item The chains $c_{1},c_{2}$ are over the same domain a-la notation \ref{chains over the same domain}: $c_{1}\stackrel{S}{\sim}c_{2}$

        \item Denoting $S = S_{c_{1}} = S_{c_{2}}$, for all functions $f\in\mathcal{F}(S,Y)$, all measures $\mu\in\Xi(\Sigma_{S},\overline{M})$, and any choice of IA-integration:
        $$\int\limits_{c_{1}}g(f,d\mu) = \int\limits_{c_{2}}g(f,d\mu)$$
        Meaning that if one integral is undefined, both are undefined, and if one is defined, both are defined and equal each other.
    \end{enumerate}
\end{definition}
\begin{notation}
    As integration equivalence is an equivalence relation, I denote the integration equivalence class of a chain $c$ as $[c]$. For reasons that'll be clarified by definition \ref{integrable chain space definition} and theorem \ref{the integrable chain space is a unital magma}, the integration equivalence classes are to be called \textbf{$\mathcal{C}$-integrable chains}.
\end{notation}
Since integration equivalence is an equivalence relation, the topological quotient space can be defined:
\begin{definition}\label{integrable chain space definition}
    Given a topological space $X$ and an integrable basis $\mathcal{B}$ formed on $X$ by an integrable collection $B$ and a choice of calculus $(\mathcal{O},Y,\overline{M},\overline{G},\int g)$, let $C$ be the set of all possible chains of $\mathcal{B}$. The corresponding \textbf{integrable chain space} $\mathcal{C}$ is the topological quotient space of $C$ by the integration-equivalence relation.
\end{definition}
\begin{notation}
    Just like the integrable basis $\mathcal{B}$, the integrable chain space $\mathcal{C}$ is said to be formed on a topological space $X$ by an integrable collection $B$ and a choice of calculus $(\mathcal{O},Y,\overline{M},\overline{G},\int g)$.
\end{notation}
\begin{theorem}\label{the integrable chain space is a unital magma}
     An integrable chain space $\mathcal{C}$ formed on a topological space $X$ by an integrable collection $B$ and a choice of calculus $(\mathcal{O},Y,\overline{M},\overline{G},\int g)$ is a unital magma with respect to a magma operation $+_{\mathcal{C}}$ defined via $[c_{1}]+_{\mathcal{C}}[c_{2}] := \sr{c_{1}+_{\mathcal{B}}c_{2}}$, for all $\mathcal{B}$-integrable chains $c_{1},c_{2}$. The identity element of $(\mathcal{C},+_{\mathcal{C}})$ is the empty set $\emptyset$.
\end{theorem}
\begin{proof}
    Let $X$ be a topological space, and let $\mathcal{C}$ be an integrable chain space formed on $X$ by an integrable collection $B$ and a choice of calculus $(\mathcal{O},Y,\overline{M},\overline{G},\int g)$. The claims to be proven are:
    \begin{enumerate}
        \item For all $[a],[b]\in\mathcal{C}$, and all $a_{1},a_{2}\in [a],b_{1},b_{2}\in [b]$: $[a_{1}+_{\mathcal{B}}b_{1}] = [a_{2}+_{\mathcal{B}}b_{2}]$.

        \item For all $[a]\in\mathcal{C}$: $[a+_{\mathcal{B}}\emptyset]=[\emptyset +_{\mathcal{B}} a] = [a]$.
    \end{enumerate}
    To prove claim 1, first note that by definition \ref{integration equivalence definition}, $S_{a_{1}}=S_{a_{2}}$ and $S_{b_{1}}=S_{b_{2}}$. As such, it must also be that $S_{a_{1}+_{\mathcal{B}}b_{1}}=S_{a_{2}+_{\mathcal{B}}b_{2}}$. Now denoting $S_{a}=S_{a_{1}}$, $S_{b}=S_{b_{1}}$, and $S_{a+_{\mathcal{B}}b}=S_{a_{2}+_{\mathcal{B}}b_{1}}$, note that again by definition \ref{integration equivalence definition} it must be the case that for all functions $f\in\mathcal{F}(S_{a},Y),h\in\mathcal{F}(S_{b},Y)$ and all measures\\
    $\mu\in\Xi(S_{a},\overline{M}),\nu\in\Xi(S_{b},\overline{M})$, it must be that
    \begin{align*}
        \int\limits_{a_{1}}g(f,d\mu) &= \int\limits_{a_{2}}g(f,d\mu)\\
        \int\limits_{b_{1}}g(h,d\nu) &= \int\limits_{b_{2}}g(h,d\nu)
    \end{align*}
    Thus, for all functions $f\in\mathcal{F}(S_{a+_{\mathcal{B}}b},Y)$ and all measures $\mu\in\Xi(S_{a+_{\mathcal{B}}b},\overline{M})$, it must be the case that
    \begin{equation*}
        \int\limits_{a_{1}}g(f,d\mu) +_{G} \int\limits_{b_{1}}g(f,d\mu) = \int\limits_{a_{2}}g(f,d\mu) +_{G} \int\limits_{b_{2}}g(f,d\mu)
    \end{equation*}
    Taking all of this together, it results that by definitions \ref{integration on chains} and \ref{integration equivalence definition}: $[a_{1}+_{\mathcal{B}}b_{1}]=[a_{2}+_{\mathcal{B}}b_{2}]$, meaning that claim 1 is now proven. To prove claim 2, simply note that by definition \ref{measure definition}, the measure of the empty set $\emptyset$ is always $\bb{0}_{M}$. As such, due to definitions \ref{definition of simple integral} and \ref{IA integral definition}, as well as theorem \ref{theorem: integral of a simple function}, it trivially follows that $\int\limits_{\emptyset}g(f,d\mu)=\bb{0}_{G}$ for any function $f\in\mathcal{F}(\emptyset,Y)$ and any measure $\mu\in\Xi(\emptyset,\overline{M})$. As such, it follows from definitions \ref{integration on chains} and \ref{integration equivalence definition}, and the fact that $\overline{G}$ is an extended topological unital magma satisfyinh $\forall K\in\overline{G}: K +_{G} \bb{0}_{G} = \bb{0}_{G} +_{G} K = K$, that for all $[a]\in\mathcal{C}$: $[a+_{\mathcal{B}}\emptyset]=[\emptyset +_{\mathcal{B}} a] = [a]$, proving claim 2.
\end{proof}
\begin{remark}[Region-additivity of integrals, commutativity, and noncommutative integrals]\label{region additivity remark}
    Note that in the definitions \ref{measure definition}, \ref{definition of simple integral} and \ref{IA integral definition}, measure and integration (both simple and IA) have not been required to be additive over unions of disjoint domains in the topological space $X$ being measured/integrated over. For integration, this means situations of $$\int\limits_{A\union B}g(f,d\mu) \neq \int\limits_{A}g(f,d\mu) +_{G} \int\limits_{B}g(f,d\mu)$$
    with $A,B$ disjoint may be allowed. This was because (a) it turned out that property was unnecessary to define integration, and (b) the union operation of sets is commutative, but the magma operations $+_{M}$ and $+_{G}$ are allowed in this paper to be noncommutative - the operation of set union doesn't carry the necessary information for order of operations needed in the case of $+_{G}$ noncommutativity. Further note, that the requirement of integration being additive over union of integration domains can be framed in terms of integration equivalence: Given an integrable basis $\mathcal{B}$ formed on a topological space $X$ by an integrable collection $\mathcal{B}$ and a choice of calculus $(\mathcal{O},Y,\overline{M},\overline{G},\int g)$, let us say that the integral $\int g$ is \textbf{region additive} over the basic sets of $B$ if for all $\iota_{*}\in\mathcal{O}$, and all $\mathcal{B}$-integrable chains $c=\sumx{\mathcal{B}}{k=1}{n}\iota_{*}S_{k}$ such that all $S_{k}$ are pairwise disjoint, $[c] = [\iota_{*}S_{c}]$. It follows trivially that a necessary condition for region additivity is commutativity of the integration "addition" operation $+_{G}$. This doesn't mean that notions of integration based on noncommutative $+_{G}$ magma operations are ill-defined - it just means that they won't be region additive the way more familiar notions of integration usually are.
\end{remark}
\newpage
\subsection{Integrable chain complexes and disintegration of measures}
\subsubsection{Integrable chain complexes}
Now that integrable chain spaces have been defined, I can define integrable chain complexes in analogy to the traditional definition of chain complexes.
\begin{definition}\label{integrable chain complex definition}
    Given a topological space $X$, an \textbf{integrable chain complex} $\set{\mathcal{C}_{n},\partial_{n}}$ on $X$ is a sequence of integrable chain spaces $\set{\emptyset},\mathcal{C}_{0},\mathcal{C}_{1},...$, with each $\mathcal{C}_{n}$ formed by an integrable collection $B_{n}$ and a choice of calculus $(\mathcal{O}_{n},Y_{n},\overline{M}_{n},\overline{G},\int g_{n})$, together with a family of \textbf{boundary operators}\\ $\partial_{n}:\mathcal{C}_{n}\to\mathcal{C}_{n-1}$ that satisfy:
    \begin{enumerate}
        \item The boundary operators are unital magma homomorphisms:
        $$\forall n\in\N,\forall [c_{1}],[c_{2}]\in\mathcal{C}_{n}:\partial_{n}\br{[c_{1}]+_{\mathcal{C}_{n}}[c_{2}]}=\partial_{n}\br{[c_{1}]} +_{\mathcal{C}_{n-1}} \partial_{n}\br{[c_{2}]}$$

        \item Given two $\mathcal{C}_{n}$-integrable chains defined over the same domain as each other, their boundaries must also defined over the same domain as each other:
        $$\forall n\in\N,\forall [c_{1}],[c_{2}]\in\mathcal{C}_{n}:[c_{1}]\stackrel{S}{\sim}[c_{2}]\Rightarrow \partial_{n}[c_{1}]\stackrel{S}{\sim}\partial_{n}[c_{2}]$$

        \item For all $\mathcal{C}_{n}$-integrable chains $[c]\in\mathcal{C}_{n}$, the base set $S_{\partial_{n}[c]}$ is contained in the closure of the base set $S_{[c]}$ under the topology of $X$:
        $$\forall n\in\N,\forall[c]\in\mathcal{C}_{n}: S_{\partial_{n}[c]}\subseteq \closure_{X}S_{[c]}$$

        \item The boundary of a boundary is empty:
        $$\forall n\in\N,\forall [c]\in\mathcal{C}_{n+1}:\br{\partial_{n}\circ\partial_{n+1}}[c]=\emptyset$$

        \item The boundary of 0-chains is empty:
        $$\forall [c]\in\mathcal{C}_{0}:\partial_{0}[c]=\emptyset$$
    \end{enumerate}
\end{definition}
Important things to note:
\begin{enumerate}
    \item Each choice of calculus $(\mathcal{O}_{n},Y_{n},\overline{M}_{n},\overline{G},\int g_{n})$ in definition \ref{integrable chain complex definition} is allowed to have different $(\mathcal{O}_{n},Y_{n},\overline{M}_{n},\int g_{n})$ for different values of $n$, but all must share the same integration codomain $\overline{G}$.

    \item Given a sequence of integrable chain spaces $\set{\emptyset},\mathcal{C}_{0},\mathcal{C}_{1},...$ as in definition \ref{integrable chain complex definition}, it's not in general guaranteed that a corresponding integrable chain complex $\set{\mathcal{C}_{n},\partial_{n}}$ can be defined - the question of whether or not suitable non-trivial boundary operators exist is crucial. If they don't, it'll be impossible to define differentiation of functions and differential forms in the manner done in this paper.
\end{enumerate}
\begin{notation}\label{k-chain notation}
    Given an integrable chain complex $\set{\mathcal{C}_{n},\partial_{n}}$ on a topological space $X$, for any $k\in\N$, the $\mathcal{C}_{k}$-integrable chains belonging to $\mathcal{C}_{k}$ are called the \textbf{$k$-chains} of $\set{\mathcal{C}_{n},\partial_{n}}$.
\end{notation}
\begin{notation}
    Following the notation used in discussion of traditional chain complexes, all integrable chain complexes are said to be \textbf{bounded below}, since they don't extend to non-empty integrable chain spaces $\mathcal{C}_{-1},\mathcal{C}_{-2},...$ . Following from this, an integrable chain complex is said to be \textbf{bounded above}, and thus fully \textbf{bounded}, if there exists some $N\in\N$ such that for all $n>N$, $\mathcal{C}_{n}=\set{\emptyset}$.
\end{notation}
\begin{notation}\label{reduced integrable chain complex}
    Given a topological space $X$, an integrable chain complex $\set{\mathcal{C}_{n},\partial_{n}}$ is said to be \textbf{reduced} if the integrable collection forming $\mathcal{C}_{0}$ is $B_{0}\subseteq X$. This corresponds to the intuition that 0-chains should represent points of the topological space $X$. Note, in this context the union of points in $X$ is undefined, so $B_{0}$ is vacuously closed under unions. Similarly, the subspace Borel $\sigma$-algebra of a point $x\in X$ is simply $\set{\emptyset,x}$.
\end{notation}
\begin{remark}[Integrable chain complexes define an integer dimension on subsets of a topological space]\label{remark: integrable chain complexes define an integer dimension}
    Note that definition \ref{integrable chain complex definition} and notation \ref{k-chain notation} suggest that integrable chain complexes on a topological space $X$ impose a sort of \textbf{integer dimension} on certain subsets of $X$ - specifically, on the base sets of integrable chains in the complex. Given the topological space $x$, this notion of integer dimension may have different geometric interpretations depending on the choice of integrable chain complex on $X$.
\end{remark}

\subsubsection{Measures and disintegrations on integrable chain complexes}
Now that I have integrable chain complexes on a topological space, I can define one last crucial ingredient needed to define differentiation of functions on integrable chain complexes - the concept of how measures on a $k$-chain relate to measures on its boundary.
\begin{definition}\label{measure disintegrations}
    Given an integrable chain complex $\set{\mathcal{C}_{n},\partial_{n}}$ formed on a topological space $X$ by integrable collections $B_{n}$ and choices of calculus $(\mathcal{O}_{n},Y_{n},\overline{M}_{n},\overline{G},\int g_{n})$, let
    $$\mathcal{U}(\mathcal{B}_{n})=\set{\br{S,\mathcal{U}(S,\overline{M}_{n})}\Big| S\in B_{n},\mathcal{U}(S,\overline{M}_{n})\subseteq \Xi(S,\overline{M}_{n})}$$
    Be a subset of the space of measures on the integrable bases $\mathcal{B}_{n}$, called the \textbf{decomposable measures} on $\mathcal{B}_{n}$. A sequence of \textbf{measure disintegrations} on $\set{\mathcal{C}_{n},\partial_{n}}$ is a sequence of mappings\\
    $\pi_{n}:\mathcal{U}(\mathcal{B}_{n})\to\mathcal{U}(\mathcal{B}_{n-1})$, relating the disintegrable measures on $\mathcal{C}$-integrable chains to corresponding measures on their boundaries:
    $$\forall [c]\in\mathcal{C}_{n},(S_{c},\mu)\in\mathcal{U}(\mathcal{B}_{n}):\pi_{n}\br{S_{[c]},\mu}\in\br{S_{\partial_{n}[c]},\mathcal{U}(S_{\partial_{n}[c]},\overline{M}_{n-1})} $$
    Where $S_{[c]}$ is the base set of $[c]$, and $S_{\partial_{n}[c]}$ is the base set of $\partial_{n}[c]$. Note, condition 2 in definition \ref{integrable chain complex definition} guarantees that $S_{\partial_{n}[c]}$ is well defined. If the $\mathcal{C}$-integrable chain $[c]$ is known from context, then I will also simply denote $\pi_{n}\mu := \pi_{n}(\mu)\in\mathcal{U}(S_{\partial_{n}[c]},\overline{M}_{n-1})$, accordingly.
\end{definition}
\begin{corollary}[The measure disintegration of a measure disintegration is the zero measure]
    Note, due to the fact that by definition \ref{integrable chain complex definition} the boundary of a boundary is empty a-la $$\partial_{n}\circ\partial_{n+1}=\emptyset$$
    It results from definitions \ref{measure definition} and \ref{measure disintegrations} that the measure disintegration of a measure disintegration is the zero measure (unital magma identity element) of the corresponding measure codomain:
    $$\forall [c]\in\mathcal{C}_{n+1},(S_{c},\mu)\in\mathcal{U}(\mathcal{B}_{n+1}):\br{\pi_{n}\circ\pi_{n+1}}\br{S_{[c]},\mu} = \br{\emptyset, \bb{0}_{M-1}} $$
\end{corollary}
\begin{notation}\label{measured chain complex}
    An integrable chain complex $\br{\mathcal{C}_{n},\partial_{n}}$ formed on a topological space $X$ by integrable collections $B_{n}$ and choices of calculus $(\mathcal{O}_{n},Y_{n},\overline{M}_{n},\overline{G},\int g_{n})$, the integrable chain complex is said to be \textbf{measured}, and denoted as $\br{\mathcal{C}_{n},\partial_{n},\pi_{n}}$, if a sequence of disintegrations $\pi_{n}:\mathcal{U}(\mathcal{B}_{n})\to\mathcal{U}(\mathcal{B}_{n-1})$ exists on it as per definition \ref{measure disintegrations}.
\end{notation}
\begin{notation}\label{notation: choosing a decomposable measures on a specific base set}
    Given a base set $S_{[c]}$ of a $\mathcal{C}_{n}$-integrable chain $[c]$ in a measured integrable chain complex $\br{\mathcal{C}_{n},\partial_{n}}$, the notation "$\mathcal{U}\br{S_{[c]},\overline{M}_{n+1}}$" denotes specifically the decomposable measures that are defined on $\Sigma_{S_{[c]}}$.
\end{notation}

\section{The Exterior Calculus on Integrable Chain Complexes}\label{The Exterior Calculus on Integrable Chain Complexes}
\subsection{Differential forms on an integrable chain space}
Given an integrable chain space $\mathcal{C}$ formed on a topological space $X$ by an integrable collections $B$ and a choice of calculus $(\mathcal{O},Y,\overline{M},\overline{G},\int g)$, consider the set of all mappings $\alpha:\mathcal{C}\to\overline{G}$, taking in $\mathcal{C}$-integrable chains and returning values in $\overline{G}$. That set may be considered the set of all differential forms on $\mathcal{C}$, and integration of a differential form $\alpha$ on a $\mathcal{C}$-integrable chain $[c]$ may be defined via $\int\limits_{[c]}\alpha := \alpha([c])$. However, that set is a bit too large to deal with, so in this section I will now form a more tractable set of differential forms, whose properties are easier to digest.
\begin{definition}\label{basic differential forms definition}
    Given an integrable chain space $\mathcal{C}$ formed on a topological space $X$ by an integrable collections $B$ and a choice of calculus $(\mathcal{O},Y,\overline{M},\overline{G},\int g)$, a \textbf{basic differential form} is a mapping $\alpha:\mathcal{C}\to\overline{G}$ that satisfies the following conditions:
    \begin{enumerate}
        \item $\forall [c_{1}],[c_{2}]\in\mathcal{C}: \alpha\br{[c_{1}]+_{\mathcal{C}}[c_{2}]} = \alpha([c_{1}]) +_{G} \alpha([c_{2}])$

        \item For all $[c]\in\mathcal{C}$, there exists a function $f\in\mathcal{F}(S_{[c]},Y)$ and a measure $\mu\in\Xi(S_{[c]},\overline{M})$\\
        such that $\alpha([c]) = \int\limits_{[c]}g(f,d\mu)$.\\
        The function $f$ and measure $\mu$ taken together are called a \textbf{representation} of $\alpha$ on $[c]$.

        \item For all $[a],[b]\in\mathcal{C}$ such that $S_{[a]}\subseteq S_{[b]}$, the representation of $\alpha$ on $[a]$ equals the restriction to $S_{[c]}$ of the representation of $\alpha$ on $[b]$. Meaning, if $\alpha$ is represented on $[b]$ via a function $f\in\mathcal{F}(S_{[b]},Y)$ and a measure $\mu\in\Xi(S_{[b]},\overline{M})$, then $\alpha$ is represented on $[a]$ via the restriction of $f$ to $S_{[a]}$, and the restriction of $\mu$ to $\Sigma_{S_{[a]}}$.
    \end{enumerate}
\end{definition}
\begin{definition}\label{sum of differential forms}
    Given an integrable chain space $\mathcal{C}$ formed on a topological space $X$ by an integrable collections $B$ and a choice of calculus $(\mathcal{O},Y,\overline{M},\overline{G},\int g)$, and two basic differential forms $\alpha_{1},\alpha_{2}$ on $\mathcal{C}$, the \textbf{sum} of $\alpha_{1},\alpha_{2}$ is defined as follows:
    $$\forall [c]\in\mathcal{C}:\br{\alpha_{1}+_{\Omega}\alpha_{2}}([c]) := \alpha_{1}([c]) +_{G} \alpha_{2}([c])$$
\end{definition}
\begin{definition}\label{differential form space}
    Given an integrable chain space $\mathcal{C}$ formed on a topological space $X$ by an integrable collections $B$ and a choice of calculus $(\mathcal{O},Y,\overline{M},\overline{G},\int g)$, the \textbf{differential form space} on $\mathcal{C}$ is the set $\Omega(\mathcal{C})$ of all mappings $\omega:\mathcal{C}\to\overline{G}$ that can be formed as a finite sum of basic differential forms:
    $$\exists n\in\N,\forall [c]\in\mathcal{C}: \omega([c]) = \br{\sumx{\Omega}{k=1}{n}\alpha_{k}}([c]) = \sumx{G}{k=1}{n}\alpha_{k}([c])$$
    Where each $\alpha_{k}$ is a basic differential form on $\mathcal{C}$. The elements of $\Omega(\mathcal{C})$ are called the \textbf{differential forms} on $\mathcal{C}$.
\end{definition}
\begin{theorem}
    Given an integrable chain space $\mathcal{C}$ formed on a topological space $X$ by an integrable collections $B$ and a choice of calculus $(\mathcal{O},Y,\overline{M},\overline{G},\int g)$, the differential form space $\Omega(\mathcal{C})$ defined in definition \ref{differential form space} is a unital magma under the sum operation defined in \ref{sum of differential forms}, with an identity element $\bb{0}_{\Omega}$ defined via $\forall [c]\in\mathcal{C}:\bb{0}_{\Omega}([c]) = \bb{0}_{G}$
\end{theorem}
The proof is trivial, as the sum of any two differential forms is itself a finite sum of basic differential forms and is thus a differential form as well, and $\bb{0}_{\Omega}$ is a basic differential form easily seen to be the identity element via $\forall [c]\in\mathcal{C},\omega\in\Omega(\mathcal{C}): \br{\omega+_{\Omega}\bb{0}_{\Omega}}([c]) = \omega([c]) +_{G} \bb{0}_{\Omega}([c]) = \omega([c]) +_{G} \bb{0}_{G} = \omega([c])$, and much the same with $\br{\bb{0}_{\Omega}+_{\Omega}\omega}([c]) - \omega([c])$.

\subsection{The exterior derivative and integrable cochain complexes}
Now that I have differential form spaces, I can easily define the exterior derivative by taking Stokes' theorem as an axiom, and form the resulting integrable cochain complex.
\begin{definition}\label{exterior derivative operators definition}
     Given an integrable chain complex $\set{\mathcal{C}_{n},\partial_{n}}$ formed on a topological space $X$ by integrable collections $B_{n}$ and choices of calculus $(\mathcal{O}_{n},Y_{n},\overline{M}_{n},\overline{G},\int g_{n})$, and given the differential form spaces $\Omega_{n}$ on $\set{\mathcal{C}_{n},\partial_{n}}$, the \textbf{explicit exterior derivative operators} are a sequence of operators\\ $d_{n}:B_{n+1}\times\Omega_{n}\to\Omega_{n+1}$ defined as follows for all $n\in\N$:\\
     For all $S\in B_{n+1}$ and all $\omega\in\Omega_{n}$:
        \begin{equation}\label{generalized Stokes theorem as axiom - differential forms}
             \forall [c]\in\mathcal{C}_{n+1}: S_{[c]}\subseteq S\Rightarrow d_{n}\sr{S,\omega}([c]) := \omega\br{\partial_{n+1}[c]}
        \end{equation}
    Where if $S_{[c]}\nsubseteq S$, then $d_{n}[S,\omega]([c])$ can be anything.
\end{definition}
Note, equation \ref{generalized Stokes theorem as axiom - differential forms} is in fact the generalized Stokes theorem, taken here as axiom to define the exterior derivative operators.
\begin{corollary}\label{requirement for exterior-differentiability corollary}
    In the context of definition \ref{exterior derivative operators definition}, the requirement needed for $d_{n}\sr{S,\omega}([c])$ to be defined on a chain $[c]\in\mathcal{C}_{n+1}$ such that $S_{c}\subseteq S$ is that $\omega\br{\partial_{n+1}[c]}$ is a well defined value in $\overline{G}$. 
\end{corollary}
\begin{notation}
    In the context of definition \ref{exterior derivative operators definition} and corollary \ref{requirement for exterior-differentiability corollary}, given a basic set $S\in B_{n+1}$, it is denoted that a differential form $\omega\in\Omega_{n}$ is \textbf{exterior differentiable} on $S$ if for all $[c]\in\mathcal{C}_{n+1}$ such that $S_{c}\subseteq S$, $\omega\br{\partial_{n+1}[c]}$ is defined in $\overline{G}$. If a differential form $\omega\in\Omega_{n}$ is differentiable on all basic sets $S\in B_{n+1}$, then it's \textbf{differentiable everywhere}.
\end{notation}
\begin{corollary}
    For all $S,T\in B_{n+1}$ such that $S\subseteq T$, all differential forms $\omega\in\Omega_{n}$ that are exterior-differentiable on $T$, and all $[c]$ such that $S_{[c]}\subseteq S$:
         $$d_{n}\sr{S,\omega}([c]) = \omega\br{\partial_{n+1}[c]} = d_{n}\sr{T,\omega}([c])$$
\end{corollary}
\begin{definition}
    Given an integrable chain complex $\set{\mathcal{C}_{n},\partial_{n}}$ formed on a topological space $X$ by integrable collections $B_{n}$ and choices of calculus $(\mathcal{O}_{n},Y_{n},\overline{M}_{n},\overline{G},\int g_{n})$, given the differential form spaces $\Omega_{n}$ on $\set{\mathcal{C}_{n},\partial_{n}}$, and given a sequence of explicit exterior derivative operators $d_{n}:B_{n+1}\times\Omega_{n}\to\Omega_{n+1}$, the \textbf{exterior derivative operators} are a sequence of operators $d_{n}:\Omega_{n}\to\Omega_{n+1}$ defined as follows:
    $$\forall\omega\in\Omega_{n},[c]\in\mathcal{C}_{n+1}: d_{n}\omega([c]) := d_{n}\sr{S_{[c]},\omega}([c])$$
\end{definition}
\begin{corollary}(The exterior derivative of an exterior derivative is trivial)
    Given an integrable chain complex $\set{\mathcal{C}_{n},\partial_{n}}$ formed on a topological space $X$ by integrable collections $B_{n}$ and choices of calculus $(\mathcal{O}_{n},Y_{n},\overline{M}_{n},\overline{G},\int g_{n})$, given the differential form spaces $\Omega_{n}$ on $\set{\mathcal{C}_{n},\partial_{n}}$, and given a sequence of exterior derivative operators $d_{n}:\Omega_{n}\to\Omega_{n+1}$, the exterior derivative of an exterior derivative is trivial:
    \begin{align*}
        \forall [c]\in\mathcal{C}_{n+2},\omega\in\Omega_{n}:\br{d_{n+1}\circ d_{n}}\omega([c]) &= d_{n}\omega\br{\partial_{n+2}[c]}\\
        &= \omega\br{\br{\partial_{n+1}\circ\partial_{n+2}}[c]}\\
        &\stackrel{*}{=} \omega(\emptyset)\\
        &\stackrel{**}{=} \bb{0}_{G}
    \end{align*}
    Where in (*) I used definition \ref{integrable chain complex definition}, and in (**) I used definitions \ref{measure definition},\ref{basic differential forms definition}, and \ref{differential form space}.
\end{corollary}
\begin{definition}
    Given an integrable chain complex $\set{\mathcal{C}_{n},\partial_{n}}$ formed on a topological space $X$ by integrable collections $B_{n}$ and choices of calculus $(\mathcal{O}_{n},Y_{n},\overline{M}_{n},\overline{G},\int g_{n})$, the collection $\br{\Omega_{n},d_{n}}$ of differential form spaces taken together with the exterior derivative operators is said to be the \textbf{integrable cochain complex} formed on $X$ by integrable collections $B_{n}$\\
    and choices of calculus $(\mathcal{O}_{n},Y_{n},\overline{M}_{n},\overline{G},\int g_{n})$.
\end{definition}

\section{The Function Calculus on Integrable Chain Complexes}\label{The Function Calculus on Integrable Chain Complexes}
To consider differentiation of functions defined on integrable chain complexes, I can now use the exterior derivatives of basic differential forms, together with the concept of a measured integrable chain complex. While usually differentiation of functions is first defined, and used to define disintegration of measures as is the purview of disintegration theorems a-la Pachl\cite{pachl1978disintegration}, my approach will be going the other way around; Using measure disintegrations to define the differentiation of functions. First, however, some notation I'll need to move forward:
\begin{notation}
    An integrable chain complex $\set{\mathcal{C}_{n},\partial_{n}}$ formed on a topological space $X$ by integrable collections $B_{n}$ and choices of calculus $(\mathcal{O}_{n},Y_{n},\overline{M}_{n},\overline{G},\int g_{n})$ is said to be \textbf{f-monovalued} if $Y_{n}=Y$ are the same function codomain for all $n\in\N$.
\end{notation}
Now, to define the derivatives of functions:
\begin{definition}\label{derivative of functions definition}
    Given a measured f-monovalued integrable chain complex $\set{\mathcal{C}_{n},\partial_{n},\pi_{n}}$ formed on a topological space $X$ by integrable collections $B_{n}$ and choices of calculus $(\mathcal{O}_{n},Y,\overline{M}_{n},\overline{G},\int g_{n})$, and given the closure $\overline{S} := \closure_{X}S$ of a basic set $S\in B_{n+1}$, a function $D_{n}f\in\mathcal{F}\br{\overline{S},Y}$ is said to be a \textbf{$D_{n}$-derivative} of a function $f\in\mathcal{F}\br{\overline{S},Y}$ if the following relation holds:
    \begin{equation}\label{generalized Stokes theorem as axiom - functions}
        \forall [c]\in C_{n+1}\br{\overline{S}},\mu\in\mathcal{U}\br{S_{[c]},\overline{M}_{n+1}}: \int\limits_{[c]}g_{n+1}\br{D_{n}f,\mu} = \int\limits_{\partial_{n+1}[c]}g_{n}\br{f,d\pi_{n+1}\mu}
    \end{equation}
    Where $C_{n+1}\br{\overline{S}}$ is the set of all $\mathcal{C}_{n+1}$-integrable chains that satisfy $S_{[c]}\subseteq S$,\\
    $\mathcal{U}\br{S_{[c]},\overline{M}_{n+1}}$ is the set of decomposable measures on $S_{[c]}$ as denoted in \ref{notation: choosing a decomposable measures on a specific base set},\\
    the integral $\int\limits_{[c]}g_{n+1}\br{D_{n}f,\mu}$ is done with $D_{n}f$ restricted to $S_{[c]}$,\\
    and the integral $\int\limits_{\partial_{n+1}[c]}g_{n}\br{f,\pi_{n+1}\mu}$ is done with $f$ restricted to $S_{\partial_{n+1}[c]}$.
\end{definition}
Note, just like eq. \ref{generalized Stokes theorem as axiom - differential forms} in definition \ref{exterior derivative operators definition}, eq. \ref{generalized Stokes theorem as axiom - functions} in definition \ref{derivative of functions definition} is once again the generalized Stokes theorem taken as axiom, this time to define the derivative of a function.
\begin{definition}
    Given a measured f-monovalued integrable chain complex $\set{\mathcal{C}_{n},\partial_{n},\pi_{n}}$ formed on a topological space $X$ by integrable collections $B_{n}$ and choices of calculus $(\mathcal{O}_{n},Y,\overline{M}_{n},\overline{G},\int g_{n})$,\\
    let $\mathcal{S}_{n} := \set{\closure_{X}S\Big| S\in B_{n}}$ be the set of $X$-closures of the basic sets of $B_{n}$.\\
    A family of \textbf{derivative operators} $$D_{n}:\br{\union\limits_{S\in\mathcal{S}_{n}}\mathcal{F}\br{S,Y}}\to\br{\union\limits_{S\in\mathcal{S}_{n+1}}\mathcal{F}\br{S,Y}}$$
    is defined on $\set{\mathcal{C}_{n},\partial_{n},\pi_{n}}$ as follows:\\
    For all $S\in\mathcal{S}_{n}$, and all $f\in\mathcal{F}\br{\overline{S},Y}$, if $D_{n}f$ is defined, it is a $D_{n}$-derivative of $f$ on $S$ as per definition \ref{derivative of functions definition}.
\end{definition}
\begin{notation}[The ply and order of derivatives - differentiability classes]
     In this context, a function $f\in\mathcal{F}\br{\overline{S},Y}$ for which $D_{n}f$ is defined is to be called a\\
     \textbf{$D_{n}$-differentiable} function on $\overline{S}$. For each $D_{n}$, the integer subscript $n$ is to be called the \textbf{ply} of the derivative. This is to distinguish the ply from the unrelated notion of the \textbf{order} of a derivative, which is to refer to the amount of times $D_{n}$ is applied to a function $f\in\mathcal{F}\br{\overline{S},Y}$.\\
     For example, $D_{n}^{2}f := \br{D_{n}\circ D_{n}}f$ is the \textbf{second order $D_{n}$-derivative} of $f$. \textbf{$D_{n}$-Differentiability classes} can then be defined, as sets of functions that have $k$th order $D_{n}$-derivatives.
\end{notation}
\begin{corollary}[Connection to exterior calculus]\label{Connection to exterior calculus}
     Given a measured f-monovalued integrable chain complex $\set{\mathcal{C}_{n},\partial_{n},\pi_{n}}$ formed on a topological space $X$ by integrable collections $B_{n}$ and choices of calculus $(\mathcal{O}_{n},Y,\overline{M}_{n},\overline{G},\int g_{n})$, and given a function $f\in\mathcal{F}\br{\overline{S},Y}$ with $\overline{S}\in\mathcal{S}_{n+1}$, consider any differential form $\alpha\in\Omega_{n}$ that satisfies
     $$\forall [c]\in\mathcal{C}_{n+1}: S_{[c]}\subseteq S \Rightarrow \alpha\br{\partial_{n+1}[c]} = \int\limits_{\partial_{n+1}[c]}g_{n}\br{f,d\pi_{n+1}\mu}$$
     For some $\mu\in\mathcal{U}\br{S_{[c]},\overline{M}_{n+1}}$ where the integral $\int\limits_{\partial_{n+1}[c]}g_{n}\br{f,d\pi_{n+1}\mu}$\\
     is done with $f$ restricted to $\partial_{n+1}[c]$.\\
     Now consider the basic differential form $\alpha'\in\Omega_{n+1}$ that satisfies
     $$\forall [c]\in\mathcal{C}_{n+1}: S_{[c]}\subseteq S \Rightarrow \alpha'([c]) = \int\limits_{[c]}g_{n+1}\br{D_{n}f,\mu}$$
     Where the integral $\int\limits_{[c]}g_{n+1}\br{D_{n}f,\mu}$ is done with $f$ restricted to $S_{[c]}$.\\
     Clearly, $\alpha' = d_{n}\sr{\overline{S},\alpha}$ is the exterior derivative of $\alpha$ on $\overline{S}$.
\end{corollary}
From here on, given a measured f-monovalued integrable chain complex $\set{\mathcal{C}_{n},\partial_{n},\pi_{n}}$ formed on a topological space $X$ by integrable collections $B_{n}$ and choices of calculus $(\mathcal{O}_{n},Y,\overline{M}_{n},\overline{G},\int g_{n})$, further notions of function calculus can be easily defined for the derivative operators $D_{n}$ of different ply $n$, each defined on functions $f\in\mathcal{F}\br{\overline{S}_{n},Y}$, where $\overline{S}_{n}\in\mathcal{S}_{n+1}$.
\begin{remark}[$D_{n}$-differentiation as a generalization of $n$-variables differentiation]\label{Dn-differentiation as a generalization $n$-variables differentiation}
    Note, in the standard exterior calculus, 1-chains correspond to 1D curves, making the $D_{0}$ operator equivalent to single-variable differentiation. 2-chains correspond to 2D surfaces, making $D_{1}$ a 2-variables differential operator. Going on with this logic, it results that:\\
    \textit{The above function calculus defines a differential operator $D_{n}$, where the ply $n$ serves as a generalization to $(n+1)$-dimensional domains of functions.} This conlusion can be seen as a direct result of remark \ref{remark: integrable chain complexes define an integer dimension}.
\end{remark}

\section{Conclusions}
This paper has done quite a lot, so allow me to summarize:
\begin{conclusion}
    Section \ref{A Topological Construction of Integration} defined a highly abstract notion of integration, of functions taking value in general topological unital magmas against measures taking value in general extended (compactified) topological unital magmas, on the Borel $\sigma$-algebra of a general topological space. The Riemann integral, geometric product integral, and Lebesgue integral have all been shown as special cases of the IA integration defined in this paper.
\end{conclusion}
\begin{conclusion}
    Section \ref{A Topological Construction of Integrable Chain Complexes} defined notions reminiscent of the traditional definition of chains in a chain complex, specifically suited to integration that may be noncommutative and nonassociative. Integration on these "integrable chains" was defined, and it's been shown that the collection of such integrable chains, aka the integrable chain space, is endowed with a unital magma structure by the choice of calculus on a topological space. Furthermore, \textbf{integrable chain complexes} and further constructions on them were also defined, to accommodate the exterior calculus and function calculus defined in sections \ref{The Exterior Calculus on Integrable Chain Complexes} and \ref{The Function Calculus on Integrable Chain Complexes} after.
\end{conclusion}
\begin{conclusion}
    Also of note in section \ref{A Topological Construction of Integrable Chain Complexes} are remark \ref{orientations remark} on a conclusion as to the nature of orientations of sets in a chain, remark \ref{region additivity remark} on a certain property of noncommutative integrals, and remark \ref{remark: integrable chain complexes define an integer dimension} as to how integrable chain complexes define an integer dimension on subsets of a topological space.
\end{conclusion}
\begin{conclusion}
    Section \ref{The Exterior Calculus on Integrable Chain Complexes} defined notions of differential forms, exterior derivatives, and integrable cochain complexes for this general setting of integration based in unital magma operations, generalizing standard theory on the topic.
\end{conclusion}
\begin{conclusion}
    Section \ref{The Function Calculus on Integrable Chain Complexes} has done the same for function calculus, defining the derivatives of functions defined on certain closed subsets of a topological space. In particular, a differential operator has been found reminiscent to the notion of derivatives of functions with $n$ variables, see remark \ref{Dn-differentiation as a generalization $n$-variables differentiation}.
\end{conclusion}
In doing all of the above, this paper is hoping to generalize standard theories of integration, function calculus, and exterior calculus, in a fairly general manner that as far as the author is aware, has yet to be explored in literature. It is not a complete theory, as questions of existence, uniqueness, and well-definiteness have not been explored due to time constraints. It is, however, the skeleton of one, and can still be used to derive some fairly general conclusions, as well as a helpful start to defining specific novel forms of calculus in practical applications.

\section{Further work}\label{further work}
There are several options for further work on this topic:
\begin{enumerate}
    \item Fleshing out the theory outlined in this paper:
    \begin{itemize}
        \item What are the conditions needed to for an integrable chain complex to be well defined on a topological space, given the integrable collections and choices of calculus on it? What are the conditions for boundary operators defined in section \ref{A Topological Construction of Integrable Chain Complexes} to exist, in particular?

        \item Given a measured f-monovalued integrable chain complex, what are the conditions needed for existence and uniqueness of the derivatives of functions defined in section \ref{The Function Calculus on Integrable Chain Complexes}? 
    \end{itemize}

    \item Developing the general theory further:
    \begin{itemize}
        \item Can more structures of exterior calculus be derived in this general setting, or are additional assumptions needed for it?

        \item What does the chain rule look like in this setting? Can structures such as the codifferential, Laplacian, Hodge star, etc, be defined? Can notions of orthogonality and other key structures be described in this general setting, perhaps through some generalization of integral inner products?
    \end{itemize}
    
    \item Since differentiable functions and differentiability classes have been successfully defined in this paper, can diffeomorphisms be? If yes, can this be used to generalize notions of differentiable manifolds to more novel notions of local calculus?

    \item Duality between functions and measures - if the same basic differential form can be represented by different function-measure pairs, can that be used to find useful dualities between functions and measures? Can the choices of calculus be varied, keeping only the integration codomain fixed, to transform between different kinds of calculus? For example, can a duality between a calculus of rough functions on smooth domains and a calculus of smooth functions on rough domains be reached?

    \item Attempting to adapt J. Harrison\cite{harrison2015operator}'s ideas of taking limits of chains to the setting of this theory -  Harrison successfully used limits of chains\footnote{Recall, both my definition of integrable chains and the standard definition of chains in a chain complex require them to be made of a finite sum of oriented sets - J. Harrison\cite{harrison2015operator} found that further structures of exterior calculus can be derived, under certain conditions, by taking the limit of these sums from finite to infinite.} to derive in a unified setting the exterior calculi of many topological spaces embedded in $\R^{n}$, including highly singular and nonsmooth spaces including discrete calculus and calculus on fractals. Can a wider notion be reached via limits of integrable chains, of induced exterior calculi on topological spaces embedded in a topological space some form of calculus was defined on via the theory in this paper? 

    \item Specific applications, defining integration and exterior calculus on novel topological spaces, defining novel forms of differential and exterior calculus in general, and more. As the references of this paper imply, novel forms of calculus are aplenty, and it is my belief that the theory developed here may be applicable in some capacity to many of them, as well as to generalizing and gaining new perspectives on quite a few existing theories built on top of differential calculus.
\end{enumerate}

\printbibliography

\end{document}